\documentclass[preprint,12pt]{elsarticle}

\usepackage{amssymb}
\usepackage{multirow}
\usepackage{amsthm}
\usepackage{graphics}
\usepackage{epsfig}
\usepackage{amssymb}
\usepackage{graphicx}
\usepackage{subfigure}
\usepackage{color}
\usepackage{tikz}
\usetikzlibrary{patterns}
\usepackage{hyperref}
\usepackage{a4wide}
\usepackage{amsmath}
\usepackage{amsfonts}
\usepackage{enumerate}
\newcommand{\pf}{\noindent {\bf Proof: }}

\newtheorem*{theoremaux}{Theorem \theoremauxnum}
\gdef\theoremauxnum{1}

\newenvironment{theoremff}[2][]{%
	\def\theoremauxnum{\ref{#2}}
	\begin{theoremaux}[#1]
	}{%
\end{theoremaux}
}

\newtheorem{lemma}{\bf Lemma}[section]
 
\newtheorem{theorem}{\bf Theorem}[section]
\newtheorem{proposition}[lemma]{\bf Proposition}

\newtheorem{definition}{\bf Definition}[section]
\newtheorem{remark}{\bf Remark}[section]

\journal{~}

\begin{document}

\begin{frontmatter}



\title{On Connectivity of Comaximal Subgroup Graph}



\author{Angsuman Das\corref{cor1}}
\ead{angsuman.maths@presiuniv.ac.in}
\author{Arnab Mandal}
\ead{arnab.maths@presiuniv.ac.in}
\author{Labani Sarkar}
\ead{labanis890@gmail.com }
\address{Department of Mathematics, Presidency University, Kolkata, India}

\cortext[cor1]{Corresponding author}

\begin{abstract}
The co-maximal subgroup graph $\Gamma(G)$ of a finite group $G$ is defined to be a graph with the set of all non-trivial proper subgroups of $G$ as the set of vertices and two distinct vertices $H$ and $K$ are adjacent if and only if $HK=G$. The deleted co-maximal subgroup graph of $G$, denoted by $\Gamma^*(G)$, is defined as the graph obtained by removing the isolated vertices from $\Gamma(G)$. In this paper, we prove that for any finite group $G$, $\Gamma^*(G)$ is connected. Furthermore, we show that $\Gamma^*(G)$ either contains a cycle or is a star. When $\Gamma^*(G)$ contains a cycle, its girth is either $3$ or $4$. Finally, we classify all finite groups $G$ for which $\Gamma^*(G)$ is a star.
\end{abstract}


\end{frontmatter}

\section{Introduction}
Associating a graph to an algebraic structure and studying the algebraic properties via graph‑theoretic methods is a well‑established line of research; see \cite{Cameron-survey} for an extensive survey. Among such constructions, the {\it comaximal subgroup graph} of a finite group was introduced in \cite{akbari} and has since attracted considerable attention \cite{1st-paper,2nd-paper,dn,arnab-comaximal,das-mandal-number-of-subgroups,zn,wei-etal}.

\begin{definition}\cite{akbari} Let $G$ be a group and $S$ be the collection of all non-trivial proper subgroups of $G$. The co-maximal subgroup graph $\Gamma(G)$ of a group $G$ is defined to be a graph with $S$ as the set of vertices and two distinct vertices $H$ and $K$ are adjacent if and only if $HK=G$. The deleted co-maximal subgroup graph of $G$, denoted by $\Gamma^*(G)$, is defined as the graph obtained by removing the isolated vertices from $\Gamma(G)$.
\end{definition}

\subsection{Our Contribution}
It was shown in \cite{1st-paper} that if $G$ has a maximal subgroup which is normal in $G$, then $\Gamma^*(G)$ is connected. Although some necessary and sufficient conditions of connectedness of $\Gamma^*(G)$ exists in literature, for instance Theorem 2.2 \cite{2nd-paper}, Theorem 3.9 \cite{wei-etal}, the general conjecture that {\it $\Gamma^*(G)$ is connected for all finite groups $G$. } has remained open. Our first main result settles this conjecture in the affirmative. 
\begin{theoremff}{Gammastar_connected}
    Let $G$ be a finite group. Then $\Gamma^*(G)$ is connected.
\end{theoremff}
Once the connectivity is guaranteed, the next question is whether $\Gamma^*(G)$ contains a cycle. We prove that exactly one of the following holds (Proposition \ref{Gamma*_tree=star} and  Proposition \ref{cycle-lemma}): either $\Gamma^*(G)$ contains a cycle of length $3$ or $4$, or $\Gamma^*(G)$ is a star. Moreover, we characterise the groups for which the latter occurs in Section \ref{star-section}. The classification of groups for which $\Gamma^*(G)$ is a star is given separately for nilpotent, solvable, supersolvable, and non‑solvable groups.

\section{Connectivity of $\Gamma^*(G)$}\label{connectivity-section}
In this section, we prove the main result of the paper regarding connectedness of $\Gamma^*(G)$. We deduce a few lemmas before delving into the actual proof.

\begin{lemma}\label{no-normal-perfect}
    Let $G$ be a finite group. Then $G$ is perfect if and only if no maximal subgroup of $G$ is normal.
\end{lemma}
\pf Let $G$ be perfect and if possible, $M$ be a maximal subgroup of $G$ which is normal in $G$. Then $G/M$ is of prime order, and hence $G'\leq M <G$, a contradiction. 

Conversely, let $G$ be such that no maximal subgroup of $G$ is normal. Suppose $G'$ is a proper subgroup of $G$. Then $G/G'$ is non-trivial abelian group and hence $G/G'$ has a normal subgroup $M/G'$ of prime index. This implies that $M$ is a maximal subgroup of $G$ and $M\lhd G$, a contradiction. Thus $G'=G$. \qed

\begin{lemma}
    Let $G$ be a finite group and $N\lhd G$. Then $\Gamma(G/N)$ is an induced subgraph of $\Gamma(G)$.
\end{lemma}
\pf Define $\psi: G/N\rightarrow G$ by $\psi(H/N)=H$ where $H$ is a subgroup of $G$ containing $N$. Clearly $\psi$ is a graph isomorphism from $G/N$ to $\psi(G)$ and hence the lemma follows.\qed

Although the next lemma is proved in Theorem 3.9 \cite{wei-etal}, we give an independent and self contained proof of the same.
\begin{lemma}\label{Phi(G)-NASC}
    Let $G$ be a finite group and $\Phi(G)$ be its Frattini subgroup. Then $\Gamma^*(G)$ is connected if and only if $\Gamma^*(G/\Phi(G))$ is connected.
\end{lemma}
\pf Let $\Gamma^*(G/\Phi(G))$ be connected and $H,K\in \Gamma^*(G)$. Then there exists maximal subgroups $M_H$ and $M_K$ of $G$ such that $H \sim M_H$ and $K\sim M_K$ in $\Gamma(G)$. Clearly $\Phi(G)\subseteq M_H,M_K$ and as $\Gamma^*(G/\Phi(G))$ is connected, there exists a path $M_H/\Phi(G)\sim X_1/\Phi(G) \sim X_2/\Phi(G) \sim \cdots \sim X_t/\Phi(G)\sim M_K/\Phi(G)$ joining $M_H/\Phi(G)$ and $M_K/\Phi(G)$ in $\Gamma^*(G/\Phi(G))$. Thus $H\sim M_H\sim X_1\sim X_2\sim \cdots \sim X_t\sim M_K\sim K$ is a path joining $H$ and $K$ in $\Gamma(G)$.

Conversely, let $\Gamma^*(G)$ be connected and $U/\Phi(G),V/\Phi(G) \in \Gamma^*(G/\Phi(G))$. Therefore, $U/\Phi(G)\sim U_1/\Phi(G)$ and $V/\Phi(G)\sim V_1/\Phi(G)$ in $\Gamma^*(G/\Phi(G))$. Thus $UU_1=VV_1=G$, i.e., $U\sim U_1$ and $V\sim V_1$ in $\Gamma(G)$, i.e., $U,V \in \Gamma^*(G)$. As $\Gamma^*(G)$ is connected, there exists a path $U\sim X_1\sim X_2\sim \cdots \sim X_t\sim V$ in $\Gamma^*(G)$. For each $i \in \{1,2,\ldots,t\}$, do the following:
\begin{itemize}
    \item if $\Phi(G)\subseteq X_i$, keep $X_i$ unchanged, and
    \item if $\Phi(G)\not\subseteq X_i$, replace $X_i$ by a maximal subgroup $M_i$ containing $X_i$. (Note that $\Phi(G)\subseteq M_i$.)
\end{itemize}
So, without loss of generality, we can assume $\Phi(G)\subseteq X_i$ for all $i$ and hence we get a path $U/\Phi(G)\sim X_1/\Phi(G)\sim X_2/\Phi(G)\sim \cdots \sim X_t/\Phi(G)\sim V/\Phi(G)$ in $\Gamma^*(G/\Phi(G))$.\qed

 Now, we are in a position to prove the main result of this section. 
\begin{theorem}\label{Gammastar_connected}
    Let $G$ be a finite group. Then $\Gamma^*(G)$ is connected.
\end{theorem}
\pf If $G$ has a maximal subgroup $M$ which is normal in $G$, then $\Gamma^*(G)$ is connected. If $G$ has no maximal subgroup which is normal in $G$, then by Lemma \ref{no-normal-perfect}, $G$ must be perfect. Suppose $G$ be a perfect group of minimum order for which the theorem does not hold, i.e., $\Gamma^*(G)$ is not connected. Then $\Phi(G)$ must be trivial, as otherwise it would contradict Lemma \ref{Phi(G)-NASC}.

{\it Claim 1:} $G$ is simple.\\
{\it Proof of Claim 1:} Suppose $N$ be a minimal normal subgroup of $G$. Since $\Phi(G)$ is trivial, there exists at least one maximal subgroup $M$ which does not contain $N$ and hence $M\sim N$ in $\Gamma(G)$, i.e., $N \in \Gamma^*(G)$. Now, as $G$ is perfect, $G/N$ is also perfect and by minimality of $|G|$, $\Gamma^*(G/N)$ is connected. Let $L/N\in \Gamma^*(G/N)$ be a fixed vertex where $L$ is a maximal subgroup of $G$. Then $L\in \Gamma^*(G)$. Now, let $H$ be an arbitrary vertex in $\Gamma^*(G)$. Thus there exists $K$ such that $H\sim K$ in $\Gamma^*(G)$, i.e., $HK=G$. If $H\sim N$, let $M_1$ be a maximal subgroup containing $H$, then $H\sim N\sim M_1\sim L$ (last adjacency occurs as $N\leq L$). Similarly, if $K\sim N$, we get $H\sim K\sim N\sim M_2\sim L$ where $M_2$ is a maximal subgroup of $G$ containing $K$. If none of $H$ and $K$ are adjacent to $N$, then we have $HN\neq G$ and $KN\neq G$. Also $(HN)(KN)=HKN=G$, i.e., $HN\sim KN$ in $\Gamma(G)$ and $HN/N\sim KN/N$ in $\Gamma(G/N)$. Hence $KN/N \in \Gamma^*(G/N)$. As $\Gamma^*(G/N)$ is connected, there is a path joining $KN/N$ and $L/N$ in $\Gamma^*(G/N)$. Thus we get a path joining $KN$ and $L$ in $\Gamma^*(G)$. Thus we get a path joining $H$ and $L$ in $\Gamma^*(G)$ via $KN$. Hence in any case we get a path joining an arbitrary vertex $H\in \Gamma^*(G)$ and a fixed vertex $L \in \Gamma^*(G)$. Hence $\Gamma^*(G)$ is connected. But this is a contradiction. So, such $N\lhd G$ can not exist and hence $G$ is simple, thereby proving the claim.

Now, it suffices to show that $\Gamma^*_{max}(G)$ (vertices are non-conjugate maximal subgroups with at least one factorization) is connected, as by Theorem 2.2 \cite{2nd-paper}, that would imply $\Gamma^*(G)$ is connected. Let $M_1$ and $M_2$ be two arbitrary maximal subgroups of $G$ in $\Gamma^*_{max}(G)$. If $M_1M_2=G$, we are done. If not, there exists two maximal subgroups $H_1,H_2$ of $G$ such that $M_1H_1=M_2H_2=G$. 

{\it Claim 2:} Either $M_1H_2=G$ or $M_2H_1=G$.\\
{\it Proof of Claim 2:} As $G$ is a finite simple group, it belongs to one of the following classes:
\begin{itemize}
    \item If $G$ is an alternating group $A_n, n\geq 5$, from Theorem D and Corollary 5 of \cite{lps-book}, except for $n=6$, any maximal factorization of $G$ involves one $S_{n-k}\times S_k$ for some $1\leq k\leq 5$ and one $k$-homogeneous maximal subgroup. If two factorizations use different $k$ values, say $k_1<k_2$, then the $k_2$-homogeneous factor is also $k_1$-homogeneous, hence its product with the intransitive factor of the first factorization yields $G$. If both use the same $k$, then the intransitive factors are conjugate, contradicting non‑conjugacy. For $n=6$, we can check using GAP that Claim 2 holds for $A_6$.
    \item If $G$ is one of the sporadic groups and one can check directly from Table 6 of \cite{lps-book} that for any two factorizations $G=M_1H_1=M_2H_2$ where $M_1M_2\neq G$, one of the two cross-products, $M_1H_2,M_2H_1$ is equal to $G$.
    \item If $G$ is a classical simple group of Lie type, one can check that Claim 2 holds from Table $1,2,3$ and $4$ of \cite{lps-book}.
    \item If $G$ is an exceptional simple group of Lie type, similarly one can check that Claim 2 holds from Table $5$ of \cite{lps-book}.
\end{itemize}
Thus Claim 2 holds for all simple groups and $\Gamma^*_{max}(G)$ (vertices are non-conjugate maximal subgroups with at least one factorization) is connected.

\qed

\begin{remark}
    Note that in \cite{1st-paper}, it was shown that if $G$ has a maximal subgroup which is normal in $G$, then $\Gamma^*(G)$ is connected. In the above theorem, we prove the same result for an arbitrary finite group without any additional assumptions like the existence of normal maximal subgroups.
\end{remark}

Once the connectedness is established, in the next two propositions we show that either $\Gamma^*(G)$ contains a $3$-cycle or $4$-cycle, or $\Gamma^*(G)$ is a star. We first recall few standard results in finite group theory.

\begin{lemma}[Lemma 6, \cite{gazonzi}]\label{conjugate-lemma}
    Let $G$ be a finite group such that $G=HK$ for two subgroups $H,K$ of $G$. Then $G=(xHx^{-1})(yHy^{-1})$ for all $x,y \in G$.
\end{lemma}

\begin{lemma}[\cite{ore}]\label{solvable-conjugate-lemma}
    If $G$ is a solvable group and $M$ and $N$ are two maximal subgroups of $G$, then either $MN=G$ or $M$ and $N$ are conjugate in $G$.
\end{lemma}

\begin{proposition}\label{Gamma*_tree=star}
    Let $G$ be a finite group. Then $\Gamma^*(G)$ is a tree if and only if $\Gamma^*(G)$ is a star.
\end{proposition}
\begin{proof}
    Let $\Gamma^*(G)$ be a tree. It suffices to prove that distance between any two leaves is $2$. Let $H,K$ be two leaves of $\Gamma^*(G)$. If $H\sim K$ or $H\sim L\sim K$ for some $L\in \Gamma^*(G)$, then we are done. Suppose not, as $\Gamma^*(G)$ is a tree, we get a path $H\sim A\sim X_1 \sim \cdots \sim X_l \sim B\sim K$ in $\Gamma^*(G)$. As any support vertex (vertex adjacent to a leaf) in $\Gamma^*(G)$ is maximal in $G$, $A,B$ are maximal subgroups of $G$. Moreover, $A,B\lhd G$, as otherwise (by Lemma \ref{conjugate-lemma}) $H$ (and respectively $K$) will also be adjacent to other conjugates of $A$ (and respectively $B$) in $G$, contradicting that $H$ and $K$ are leaves.

    Thus we have $A\sim B$, i.e., $H\sim A\sim B\sim K$. Moreover, if $G$ has any other maximal subgroup $M$ other than $A,B$, then $M,A,B$ will form a triangle in $\Gamma^*(G)$, a contradiction. Thus $G$ has exactly two maximal subgroups and hence $G\cong \mathbb{Z}_{p^aq^b}$. However, by Theorem 3.8 \cite{1st-paper}, $\Gamma^*(G)$ has $4$-cycle, a contradiction. Thus $d(H,K)=1$ or $2$.
\end{proof}

\begin{proposition}\label{cycle-lemma}
    Let $G$ be a finite group such that $\Gamma^*(G)$ has a cycle. Then girth of $\Gamma^*(G)$ is either $3$ or $4$.
\end{proposition}
\begin{proof}
    Since $\Gamma^*(G)$ has at least one edge, the number of maximal subgroups of $G$ is at least $2$. If $G$ has exactly two maximal subgroups, then $G\cong \mathbb{Z}_{p^aq^b}$, where $p,q$ are distinct primes and it can be easily seen that girth of $\Gamma^*(G)$ is $4$. So, we assume that $G$ has at least three maximal subgroups.

    If no maximal subgroup of $G$ is normal in $G$, consider an edge $H\sim K$ in $\Gamma^*(G)$. Without loss of generality, we can assume $H$ and $K$ are maximal subgroups of $G$. As $H,K$ are not normal in $G$, there exists $x,y\in G$ such that $xHx^{-1}\neq H$ and $yKy^{-1}\neq K$. Thus, by Lemma \ref{conjugate-lemma}, we get a cycle $H\sim K\sim xHx^{-1}\sim yKy^{-1}\sim H$ of length $4$ in $\Gamma^*(G)$.

    If at least, two maximal subgroups, say $H,K$ are normal in $G$ and $L$ be any other maximal subgroup of $G$, then we get a $3$-cycle $H\sim K\sim L\sim H$ in $\Gamma^*(G)$.

    So, we assume that $G$ has exactly one maximal subgroup $H$ which is normal in $G$. If possible, let girth of $\Gamma^*(G)\geq 5$ and $C_1\sim C_2\sim \cdots C_k\sim C_1$ be a smallest cycle $\mathcal{C}$ of length $k\geq 5$ in $\Gamma^*(G)$. If $C_i=H$ for some $i$, then as $k\geq 5$, there exists two adjacent vertices $C_j,C_{j+1}$ which are not adjacent to $C_i=H$. Thus it follows that $C_j,C_{j+1}\subseteq H$ which contradicts that they are adjacent. Thus none of the $C_i$'s coincide with $H$. (In fact, no cycle of minimum length $k\geq 5$ can contain $H$)

    Without loss of generality, we can assume that two consecutive vertices of $\mathcal{C}$, say $C_1,C_2 \neq H$ are maximal subgroups of $G$. But this implies that $C_1,C_2,H$ forms a $3$-cycle, a contradiction.
\end{proof}

In case $G$ is solvable, we give a necessary and sufficient condition for the girth of $\Gamma^*(G)$ to be $3$ or $4$.

\begin{theorem}
    Let $G$ be a solvable group such that $\Gamma^*(G)$ has a cycle and $c$ be the number of distinct conjugacy classes of maximal subgroups of $G$. Then girth of $\Gamma^*(G)$ is $3$ if and only if $c\geq 3$.
\end{theorem}
\begin{proof}
    If $c\geq 3$, let $M_1,M_2,M_3$ be three non-conjugate maximal subgroups of $G$. Then we have $M_1\sim M_2\sim M_3\sim M_1$ (by Lemma \ref{solvable-conjugate-lemma}) giving rise to a cycle of length $3$ in $\Gamma^*(G)$. 

Conversely, let girth of $\Gamma^*(G)$ be $3$. Then without loss of generality, we get three maximal subgroups $M_1,M_2,M_3$ forming a $3$-cycle in $\Gamma^*(G)$. Clearly, these three maximal subgroups are not conjugate in $G$. Thus $c\geq 3$. 
\end{proof}

\begin{remark}
    It is to be noted that for the converse part, solvability of $G$ is not required. Moreover, it follows that if $G$ is a solvable group such that $\Gamma^*(G)$ has girth $4$, then $c=2$. However, this does not hold for non-solvable groups. For example, $\Gamma^*(A_5)$ has girth $4$ and $c\geq 3$.
\end{remark}

\section{When is $\Gamma^*(G)$ a star?}\label{star-section}
In this section, we characterize the groups in for which $\Gamma^*(G)$ is a star. It is worth noting that if $G$ is nilpotent, this has been characterized in the following theorem from \cite{1st-paper}. 

\begin{theorem}[Theorem 3.4,\cite{1st-paper}]
    If $G$ is nilpotent and $\Gamma^*(G)$ is a star, then $G$ is a cyclic group of order $p^nq$, where $p,q$ are distinct primes.
\end{theorem}
In the rest of this section, we deal with the case when $G$ is non-nilpotent.

\begin{theorem}\label{solvable_star}
    Let $G$ be a solvable group such that $\Gamma^*(G)$ is a star, then $G\cong H \rtimes \mathbb{Z}_q$, where $H$ is a $p$-group and $p,q$ are distinct primes.
\end{theorem}
\begin{proof}
    Let $H$ be the central vertex. Then it is easy to see that $H$ is a maximal subgroup of $G$ and $H\lhd G$. As a result, $[G:H]=q$, a prime. If $q$ is the only prime factor of $G$, then $G$ is nilpotent. However, by Theorem 3.4 \cite{1st-paper}, we get a contradiction. Thus $|G|$ has a prime factor, say $p$, other than $q$. Let $Q$ be a Sylow $q$-subgroup of $G$. Clearly $Q\neq H$. As $G$ is solvable, $Q$ must have a Hall complement, say $K$, i.e., $QK=G$. However, as $\Gamma^*(G)$ is a star, we must have $K=H$. Thus $H$ is a Hall subgroup of $G$ and $p\mid |H|$. If possible, let $r\neq p$ be a prime divisor of $|H|$ and $R$ be a Sylow $r$-subgroup of $G$. Then $R$ has a Hall complement, say $L$. Thus $RL=G$. Since $R\neq H$ and $\Gamma^*(G)$ is a star, this implies $L=H$, i.e., $RH=G$ a contradiction, as $R\leq H$. Thus $H$ is a Sylow $p$-subgroup of $G$. Hence the theorem follows.
\end{proof}

\begin{remark}
    In the above theorem, the $p$-group $H$ may be non-abelian, for example, if $G\cong SL(2,3)\cong Q_8 \rtimes \mathbb{Z}_3$, then $\Gamma^*(G)$ is a star. However, in the next theorem, we show that if $G$ is supersolvable, then $H$ must be a cyclic $p$-group.
\end{remark}

\begin{theorem}\label{supersolvable_star}
    Let $G$ be a supersolvable group such that $\Gamma^*(G)$ is a star, then $G\cong \mathbb{Z}_{p^n} \rtimes \mathbb{Z}_q$, where $p,q$ are distinct primes.
\end{theorem}
\begin{proof}
    As $G$ is supersolvable, by Theorem \ref{solvable_star}, $G\cong H \rtimes \mathbb{Z}_q$, where $H$ is a $p$-group and $p,q$ are distinct primes. Moreover, as $G$ is supersolvable, any maximal subgroup of $G$ is of prime index. Thus any maximal subgroup of $G$ is of order $p^n$ or $p^{n-1}q$ and $H$ is the only maximal subgroup of order $p^n$ in $G$. Let $K$ be a maximal subgroup of order $p^{n-1}q$ in $G$. Note that all subgroups of $H$ are contained in $K$, as otherwise we get an edge joining that subgroup and $K$ in $\Gamma^*(G)$, which is a star, a contradiction. Again, as $H$ is the normal Sylow $p$-subgroup of $G$, any subgroup of $G$ of order $p^{n-1}$ is contained in $H$ and hence contained in $K$. Also as $H\sim K$, we have $|H\cap K|=p^{n-1}$. Thus, it follows that $G$, and hence $H$, has exactly one subgroup of order $p^{n-1}$, making $H$ cyclic.
\end{proof}

\begin{theorem}\label{nonsolvable_star}
    Let $G$ be a non-solvable group such that $\Gamma^*(G)$ is a star. Then $G/\Phi(G)$ has a maximal subgroup $S$ such that $S$ is also minimal normal in $G$ and $S$ is a non-abelian simple group.
\end{theorem}
\begin{proof}
    By Lemma \ref{Phi(G)-NASC}, it follows that $\Gamma^*(G/\Phi(G))$ is a star. Moreover $G/\Phi(G)$ is also non-solvable with trivial Frattini subgroup. Thus it is enough to prove the result for a non-solvable group $G$ with trivial Frattini subgroup. If $H$ is the unique universal vertex in $\Gamma^*(G)$, arguing as above, one can show that $H$ is a maximal and normal subgroup of $G$.
    

    
    Let $N$ be a minimal normal subgroup of $G$ which is contained in $H$. Suppose $N\neq H$. If $N$ is not contained in some maximal subgroup $M\neq H$ of $G$, then $N\sim M$, contradicting $M$ is a leaf. Thus $N$ must be contained in all maximal subgroups of $G$, i.e., $N\subseteq \Phi(G)$. Thus $\Phi(G)$ is non-trivial, a contradiction. Thus $N=H$ is a minimal normal subgroup of $G$. Now, as $H$ is non-solvable (since $G/H$ is of prime order), $H\cong S^k$, where $S$ is a non-abelian simple group. Let $[G:H]=p$, a prime.

    Let $H\cong S^k\cong S_1\times S_2 \times \cdots \times S_k$, where $S_i\cong S$ for each $i$. Since $H\lhd G$, conjugation by an element $x \in G$ is an automorphism of $H$. Again, since $S_i$'s are the only minimal normal subgroups of $H$, the conjugation by an element $x\in G$ induces a permutation on $\Omega=\{S_1,S_2,\ldots,S_k\}$, i.e., $\varphi:G \rightarrow Sym(\Omega)$ is a group homomorphism induced by the conjugation action. Also note that any $h \in H$ induces the identity permutation on $\Omega$. Thus we get an induced homomorphism $\overline{\varphi}:G/H\rightarrow Sym(\Omega)$ given by $\overline{\varphi}(xH)=\varphi(x)$. As $|G/H|=p$, image of $\overline{\varphi}$ is a subgroup of $Sym(\Omega)$ of order dividing $p$. 

    We claim that the action given by $\overline{\varphi}$ is transitive. If not, let $A=\{S_1,S_2,\ldots,S_t\}$ be an orbit with $t<k$. Then $N=S_1\times S_2\times \cdots \times S_t$. Then $gNg^{-1}=N$ for all $g \in G$, i.e., $N\lhd G$ contradicting that $H$ is a minimal normal subgroup. Thus the action of $G/H$ given by $\overline{\varphi}$ is transitive. Thus by orbit-stabilizer theorem, $k=1$ or $p$.

    If $k=p$, there exists $x \in G\setminus H$ such that conjugation by $x$ cyclically permutes the $S_i$'s, i.e., $xS_ix^{-1}=S_{i+1}$ where addition in index is done modulo $p$.

    Let $D$ be the diagonal subgroup of $H$, i.e., $D=\{(s,s,\ldots,s):s\in S\}$. Since conjugating by $x$ permutes the coordinates cyclically and all coordinates of $D$ are equal, we have $xdx^{-1}\in D$ for all $d \in D$. Thus $D$ is normalized by $x$, and hence by all powers of $x$. Hence we have $D\langle x\rangle=\langle x\rangle D$ and $B=D\langle x \rangle$ is a proper subgroup (otherwise we get a different edge $D\sim \langle x \rangle$ in $\Gamma^*(G)$ which is a star) of $G$. Let $A=S_1\times S_2\times \cdots \times S_{p-1}$ be a proper subgroup of $H$, where $S_i\cong S$ for all $i$. Then we have $G=AB$. Now, as $A,B\neq H$, we get an inadmissible edge $A \sim B$ in $\Gamma^*(G)$ which is a star, a contradiction. Thus $k\neq p$, i.e., $k=1$ and $H=S$.
\end{proof}

\begin{remark}
    In the above theorem, if $G$ has a normal subgroup other than $S$, then it can be shown that $G/\Phi(G) \cong S\times \mathbb{Z}_p$, where $p$ is a prime and $S$ is non-factorizable (i.e., we do not have subgroups $A,B\leq S$ such that $S=AB$).
\end{remark}

\section{Conclusion and Open Issues}
In this paper we have completely resolved the question of connectivity for the deleted comaximal subgroup graph $\Gamma^*(G)$ of a finite group $G$. Furthermore, we established a structural dichotomy: either  $\Gamma^*(G)$ contains a cycle, in which case the girth is necessarily $3$ or $4$, or $\Gamma^*(G)$ is a star graph. For the latter case we provided a complete classification, covering nilpotent, solvable, supersolvable and non‑solvable groups.

Two loose ends of this work are as follows:
\begin{enumerate}
    \item It was computationally checked for groups $G$ of order upto $600$ that $diam(\Gamma^*(G))=3$. However, at present we do not have a proof of that.
    \item While we used the classification of finite simple groups and the factorisation tables of Liebeck, Praeger and Saxl to handle the simple group case, a purely group‑theoretic proof of Claim 2 (or a more elementary approach to connectivity) might still exist. Finding such a proof would be of independent interest.
\end{enumerate}

\section*{Acknowledgement}
The authors acknowledge the departmental funding of DST-FIST Sanction no. \\$SR/FST/MS-I/2019/41$. The third author is supported by UGC PhD Junior Research Fellowship, Govt. of India.

\subsection*{Statements and Declarations}
Data sharing not applicable to this article as no datasets were generated or analysed during the current study. The authors have no competing interests to declare that are relevant to the content of this article.

\end{document}